\documentclass[11pt,reqno]{amsart}

\usepackage{amsmath,amssymb,mathrsfs}
\usepackage{graphicx,cite,cases}
\newtheorem{theo}{Theorem}[section]

\newtheorem{lemm}[theo]{Lemma}

\numberwithin{equation}{section}

\begin{document}

\title[acoustic-elastic interaction problem]{An adaptive finite element PML
method for the acoustic-elastic interaction in three dimensions}

\author{Xue Jiang}
\address{School of Science, Beijing University of Posts and Telecommunications,
Beijing 100876, China.}
\email{jxue@lsec.cc.ac.cn}

\author{Peijun Li}
\address{Department of Mathematics, Purdue University, West Lafayette, IN 47907,
USA.}
\email{lipeijun@math.purdue.edu}

\thanks{The research of XJ was supported in part by China NSF grant 11401040 and
by the Fundamental Research Funds for the Central Universities 24820152015RC17.
The research of PL was supported in part by the NSF grant DMS-1151308.}



\begin{abstract}
Consider the scattering of a time-harmonic acoustic incident wave by a bounded, 
penetrable, and isotropic elastic solid, which is immersed in a homogeneous
compressible air or fluid. The paper concerns the numerical solution for such an
acoustic-elastic interaction problem in three dimensions. An exact transparent
boundary condition (TBC) is developed to reduce the problem equivalently into a
boundary value problem in a bounded domain. The perfectly matched layer (PML)
technique is adopted to truncate the unbounded physical domain into a bounded
computational domain. The well-posedness and exponential convergence of the
solution are established for the truncated PML problem by using a PML equivalent
TBC. An a posteriori error estimate based adaptive finite element method
is developed to solve the scattering problem. Numerical experiments are included
to demonstrate the competitive behavior of the proposed method.
\end{abstract}

\maketitle

\section{Introduction}

Consider the incidence of a time-harmonic acoustic wave onto a bounded,
penetrable, and isotropic elastic solid, which is immersed in a homogeneous and
compressible air or fluid. Due to the interaction between the incident wave and
the solid obstacle, an elastic wave is excited inside the solid region,
while the acoustic incident wave is scattered in the air/fluid region. This
scattering phenomenon leads to an air/fluid-solid interaction problem. The
surface of the elastic solid divides the whole three-dimensional space into a
bounded interior domain and an open exterior domain where the elastic wave and
the acoustic wave occupies, respectively. The two waves are coupled together on
the surface via the interface conditions: continuity of the normal
component of velocity and the continuity of traction. The acoustic-elastic
interaction problems have received ever-increasing attention due to their
significant applications in geophysics and seismology \cite{h-94,
h-89}. These problems have been examined mathematically by using either
variational method \cite{gl-16, glz-16} or boundary integral equation method
\cite{lm-c95, hss-na17}. Many computational approaches have also been developed
to numerically solve these problems such as boundary element method
\cite{fkw-nme06, sm-jcp06} and coupling of finite and boundary element methods
\cite{ea-nme91}. 

Since the work by B\'{e}renger \cite{b-jcp94}, the perfectly matched layer
(PML) technique has been extensively studied and widely used to
simulate various wave propagation problems, which include acoustic waves
\cite{bp-mc07, cm-sjsc98, hsz-sjma03, ls-c98, ty-anm98}, elastic waves
\cite{bpt-mc10, cxz-mc, ct-g01, hsb-jasa96, jllz-2}, and electromagnetic waves
\cite{bw-sjna05, cw-motl94}. The PML is to surround the domain of interest by a
layer of finite thickness fictitious material which absorbs all the waves coming
from inside the computational domain. It has been proven to be an effective
approach to truncated open domains in the wave computation. Combined with the
PML technique, the adaptive finite element method (FEM) has recently been
developed to solve the diffraction grating problems \cite{blw-mc10, cw-sjna03,
jllz-1} and the obstacle scattering problems \cite{cc-mc08, cw-nm08, cl-sjna05}.
Despite the large number of work done so far, they were concerned with a single
wave propagation problem, i.e., either an acoustic wave, or an elastic wave, or
an electromagnetic wave. It is very rare to study rigorously the PML problem for
the interaction of multiple waves. 

This paper aims to investigate the adaptive finite element PML method for
solving the acoustic-elastic interaction problem. An exact transparent
boundary condition (TBC) is developed to reduce the problem equivalently into a
boundary value problem in a bounded domain. The PML technique is adopted to
truncated the unbounded physical domain into a bounded computational domain.
The variational approach is taken to incorporate naturally the interface
conditions which couple the two waves. The well-posedness and exponential
convergence of the solution are established for the truncated PML problem by
using a PML equivalent TBC. The proofs rely on the error estimate between the
two transparent boundary operators. To effciently resolve the solution with
possible singularities, the a posteriori error estimate based adaptive FEM is
developed to solve the truncated PML problem. The error estimate consists of the
PML error and the finite element discretization error, and provides a
theoretical basis for the mesh refinement. Numerical experiments are reported to
show the competitive behavior of the proposed method.

The paper is organized as follows. In section 2, we introduce the model
equations for the acoustic-elastic interaction problem. In section 3,
we present the PML formulation and prove the well-posedness and convergence of
the solution for the truncated PML problem. In section 4, we discuss the
numerical implementation and show some numerical experiments. The paper is
concluded with some general remarks in section 5.

\section{Problem formulation}

In this section, we introduce the model equations for acoustic and elastic
waves, and present an interface problem for the acoustic-elastic interaction. In
addition, an exact transparent boundary condition is introduced to reformulate
the scattering problem into an boundary value problem in an bounded domain.

\subsection{Problem geometry}

Consider an acoustic plane wave incident on a bounded elastic solid which
is immersed in a homogeneous compressible air/fluid in three dimensions. The
problem geometry is shown in Figure \ref{fig:geo}. Due to the wave interaction,
an elastic wave is induced inside the solid region, while the scattered acoustic
wave is generated in the open air/fluid region. The wave propagation described
above leads to an air/fluid-solid interaction problem. The surface of the solid
divides the whole three-dimensional space into the interior domain and the
exterior domain, where the elastic wave and the acoustic wave occupies,
respectively. Let the solid $\Omega_s\subset \mathbb R^3$ be a bounded domain
with a Lipschitz boundary $\Gamma_s$. The exterior domain $\Omega_e =\mathbb
R^3\setminus\bar \Omega_s$ is assumed to be connected and filled with a
homogeneous, compressible, and inviscid air/fluid with a constant density
$\rho_a>0$. Denote by $B=\{\boldsymbol{x}=(x_1, x_2, x_3)^\top\in \mathbb R^3:
|x_j| < L_j, j=1,2,3\}$	 the rectangular box with the boundary $\partial B$,
where $L_j$ are sufficiently large such that $\bar \Omega_s\subset B$. Define
$\Omega_a = B\setminus\bar \Omega_s$. Let $\boldsymbol{n}_1$ be the unit normal
vector on $\Gamma_s$ directed from $\Omega_s$ into $\Omega_e$, and let
$\boldsymbol n_2$ be the unit outward normal vector on $\partial B$. 

\begin{figure}
\center
\includegraphics[width=0.3\textwidth]{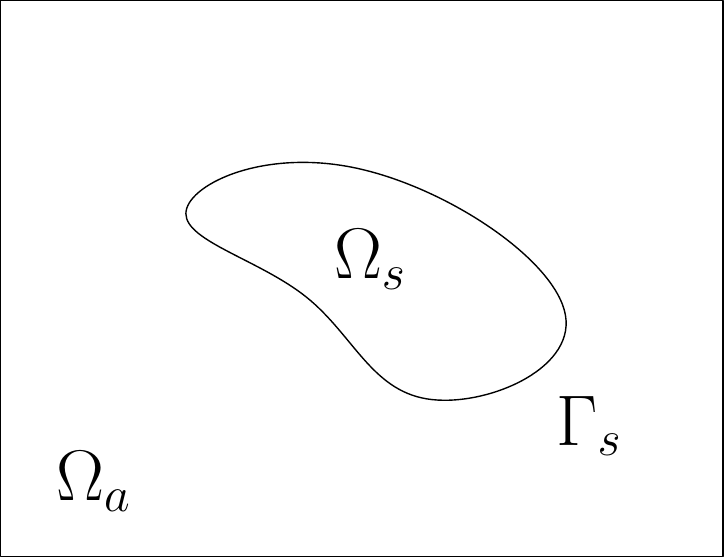}
\caption{A two-dimensional schematic of the problem geometry for the
acoustic-elastic interaction.}
\label{fig:geo}
\end{figure}

\subsection{Wave equations}

Let the elastic solid be impinged by a time-harmonic sound wave $p^{\rm
inc}$, which satisfies the three-dimensional Helmholtz equation:
\[
\Delta p^{\rm inc}+ \kappa^2 p^{\rm inc} =0\quad\text{in}~ \Omega_e, 
\]
where $\kappa=\omega/c$ is the wavenumber, $\omega>0$ is the angular frequency,
and $c$ is the speed of sound in the air/fluid. The total acoustic wave
field $p$ also satisfies the Helmholtz equation:
\begin{equation}\label{p}
\Delta p+ \kappa^2 p =0\quad\text{in}~ \Omega_e. 
\end{equation}
The total field $p$ consists of the incident field $p^{\rm inc}$ and the
scattered field $p^{\rm sc}$:
\[
 p=p^{\rm inc} + p^{\rm sc}\quad\text{in} ~ \Omega_e,
\]
where scattered field $p^{\rm sc}$ is required to satisfy the Sommerfeld
radiation condition:
\[
 \lim_{r\to\infty} r (\partial_r p^{\rm sc}-{\rm i}\kappa p^{\rm sc})=0,\quad
r=|\boldsymbol x|. 
\]

The time-harmonic elastic wave satisfies the three-dimensional Navier equation:
\begin{equation}\label{ne}
 \nabla\cdot\boldsymbol\sigma (\boldsymbol u)+\omega^2\boldsymbol
u=0\quad\text{in}~ \Omega_s,
\end{equation}
where $\boldsymbol{u}=(u_1, u_2,u_3)^\top$ is the displacement of the elastic
wave, and the stress tensor $\boldsymbol{\sigma}(\boldsymbol{u})$
is given by the generalized Hook law:
\begin{equation}\label{ghl}
\boldsymbol{\sigma}(\boldsymbol{u})=2\mu\boldsymbol{\epsilon}(\boldsymbol{u})
+\lambda{\rm tr}(\boldsymbol{\epsilon}(\boldsymbol{u}))I,\quad
\boldsymbol{\epsilon}(\boldsymbol{u})=\frac{1}{2}(\nabla\boldsymbol{u}
+\nabla\boldsymbol{u}^\top).
\end{equation}
Here $\mu(\boldsymbol x)\in L^\infty(\Omega_s), \lambda(\boldsymbol x)\in
L^\infty(\Omega_s)$ are the Lam\'{e} parameters satisfying $\mu>0, \lambda>0$,
and $\nabla\boldsymbol u$
is the displacement gradient tensor given by 
\[
 \nabla\boldsymbol{u}=\begin{bmatrix}
                       \partial_{x_1} u_1 & \partial_{x_2} u_1&\partial_{x_3}
u_1\\
                       \partial_{x_1} u_2 & \partial_{x_2} u_2&\partial_{x_3}
u_2\\
                       \partial_{x_1} u_3 & \partial_{x_2} u_3&\partial_{x_3}
u_3
                      \end{bmatrix}.
\]
Substituting \eqref{ghl} into \eqref{ne} yields 
\begin{equation}\label{une}
\nabla\cdot(\mu(\nabla\boldsymbol{u}+\nabla\boldsymbol{u}^\top))+
\nabla(\lambda\nabla\cdot\boldsymbol{u})+\omega^2\boldsymbol{u}=0
\quad\text{in}~ \Omega_s. 
\end{equation}

\subsection{Interface conditions}

To couple the acoustic wave equation and the elastic wave equation, the
kinematic interface condition is imposed to ensure the continuity of the normal
component of the velocity:
\begin{equation}\label{inc1}
\partial_{\boldsymbol n_1} p =\rho_a\omega^2
\boldsymbol{n}_1\cdot\boldsymbol{u}
\quad\text{on} ~\Gamma_s,
\end{equation}
In addition, the dynamic interface condition is required to ensure the continuity of traction:
\begin{equation}\label{inc2}
- p \boldsymbol{n}_1=\boldsymbol{\sigma}(\boldsymbol{u})\cdot\boldsymbol{n}_1
\quad\text{on} ~\Gamma_s,
\end{equation}
where $\boldsymbol\sigma(\boldsymbol u)\cdot\boldsymbol n_1$ denotes the
matrix-vector multiplication. 

\subsection{Acoustic-elastic interaction problem}

The acoustic-elastic interaction problem can be formulated into the following
coupled boundary value problem: Given $p^{\rm inc}$, to find $(p, \boldsymbol
u)$ such that 
\begin{equation}\label{bvp}
 \begin{cases}
  \Delta p+\kappa^2 p=0,\quad p=p^{\rm inc}+p^{\rm sc} &\quad\text{in} ~
\Omega_e,\\
  \nabla\cdot \boldsymbol\sigma(\boldsymbol u)+\omega^2\boldsymbol
u=0 &\quad\text{in}~\Omega_s,\\
\partial_{\boldsymbol n_1}p=\rho_a\omega^2\boldsymbol n_1\cdot\boldsymbol
u,\quad -p\boldsymbol n_1=\boldsymbol\sigma(\boldsymbol u)\cdot\boldsymbol n_1
&\quad\text{on} ~ \Gamma_s,\\
\partial_r p^{\rm sc}-{\rm i}\kappa p^{\rm
sc}=o(r^{-1})&\quad\text{as}~r\to\infty.
 \end{cases}
\end{equation}
We refer to \cite{lm-c95} for the discussion on the well-posedness of the
boundary value problem \eqref{bvp}. From now on, we assume that the
acoustic-elastic interaction problem has a unique solution. 

\subsection{Transparent boundary condition}

Given $v\in H^{1/2}(\partial B)$, we define the Dirichlet-to-Neumann (DtN)
operator $\mathscr{T}: H^{1/2}(\partial B)\to H^{-1/2}(\partial B)$ as follows:
\[
 \mathscr{T}v=\partial_{\boldsymbol n_2}u\quad \text{on} ~\partial
B,
\]
where $u$ is the solution of the exterior Dirichlet problem of the Helmholtz
equation:
\begin{equation}\label{xi}
 \begin{cases}
\Delta u+\kappa^2 u=0 &\quad\text{in} ~ \mathbb{R}^3\setminus\bar B,\\
u=v &\quad\text{on} ~ \partial B,\\
\partial_r u-{\rm i}\kappa u=o(r^{-1}) &\quad\text{as}~r\to \infty.
 \end{cases}
\end{equation}
It is well-known that the exterior problem \eqref{xi} has a unique solution
$u\in H^1_{loc}(\mathbb{R}^3\setminus\bar B)$ (cf., e.g., \cite{ck-j83}). Thus
the DtN operator $\mathscr{T}: H^{1/2}(\partial B)\to H^{-1/2}(\partial B)$ is
well-defined and is a bounded linear operator.

Using the DtN operator $\mathscr T$, we reformulate the boundary value problem
\eqref{bvp} from the open domain into the bounded domain: Given $p^{\rm inc}$,
to find $(p, \boldsymbol u)$ such that 
\begin{equation}\label{bbvp}
 \begin{cases}
\Delta p+\kappa^2 p=0 &\quad\text{in} ~ \Omega_a,\\
\nabla\cdot\boldsymbol\sigma(\boldsymbol u)+\omega^2\boldsymbol
u=0 &\quad\text{in}~\Omega_s,\\
\partial_{\boldsymbol n_1}p=\rho_a\omega^2\boldsymbol n_1\cdot\boldsymbol
u,\quad -p\boldsymbol n_1=\boldsymbol\sigma(\boldsymbol u)\cdot\boldsymbol n_1
&\quad\text{on} ~ \Gamma_s,\\
\partial_{\boldsymbol n_2} p=\mathscr{T}p+f&\quad\text{on}~\partial B,
 \end{cases}
\end{equation}
where $f=\partial_{\boldsymbol n_2} p^{\rm inc}-\mathscr{T} p^{\rm inc}$. 

To study the well-posedness of \eqref{bbvp}, we define
\[
\boldsymbol{X}:=H^1(\Omega_a)\times
H^1(\Omega_s)^3=\{\boldsymbol{\Phi}=(p,\boldsymbol{u}):p\in H^1(\Omega_a),
\boldsymbol{u}\in H^1(\Omega_s)^3\},
\]
which is endowed with the inner product:
\[
(\boldsymbol{\Phi},\boldsymbol{\Psi})_{\boldsymbol{X}}:=\int_{\Omega_a}\left(
\nabla
p\cdot\nabla \bar q + p\bar q \right){
\rm d}\boldsymbol x+\int_{\Omega_s}\left( 
\nabla\boldsymbol{u}:\nabla\bar{\boldsymbol
v}+\boldsymbol{u}\cdot\bar{\boldsymbol{v}}\right){\rm
d}\boldsymbol x
\]
for any $\boldsymbol{\Phi}=(p,\boldsymbol{u})$ and
$\boldsymbol{\Psi}=(q,\boldsymbol{v})$, where $A:B={\rm tr}(A B^\top)$ is the
Frobenius inner product of square matrices $A$ and $B$. Clearly,
$\|\cdot\|_{\boldsymbol{X}}=\sqrt{(\cdot,\cdot)_{\boldsymbol{X}}}$ is a norm on
$\boldsymbol{X}$.

Let $a: \boldsymbol{X}\times \boldsymbol{X}\to\mathbb{C}$ be the sesquilinear
form:
\begin{align}\label{asf}
 a( p,\boldsymbol u; q, \boldsymbol{v})=&\int_{\Omega_a}\left( \nabla
p\cdot\nabla \bar q-\kappa^2 p\bar q \right){\rm d}\boldsymbol
x+\int_{\Gamma_s}\rho_a\omega^2(\boldsymbol{n}_1\cdot\boldsymbol{u} )\bar
q{\rm d}s-\int_{\partial B}  (\mathscr{T}p) \bar q{\rm d}s \notag \\
 &+\int_{\Omega_s}\left(\boldsymbol\sigma(\boldsymbol
u):\nabla\bar{\boldsymbol v} -\omega^2\boldsymbol{u}\cdot\bar{\boldsymbol
v} \right){ \rm d}\boldsymbol x
 +\int_{\Gamma_s} ( p\boldsymbol{n}_1 )\cdot\bar{\boldsymbol{v}}{\rm
d}s.
\end{align}
The acoustic-elastic interaction problem \eqref{bbvp} is
equivalent to the following weak formulation: Find
$\boldsymbol{\Phi}=(p,\boldsymbol u)\in \boldsymbol{X}$ such that
\begin{equation}\label{atbc}
 a(p,\boldsymbol u; q, \boldsymbol{v})=\int_{\partial B}f \bar q{\rm d}s,\quad\forall\,
\boldsymbol{\Psi}=(q, \boldsymbol{v}) \in \boldsymbol{X}.
\end{equation}

Since we assume that the variational problem \eqref{atbc} has a unique weak
solution $(p, \boldsymbol u)\in \boldsymbol X$, the general theory in
Babu\v{s}ka and Aziz \cite[Chap. 5]{ba-73} implies that there exists a constant
$\gamma_0$ such that the following inf-sup condition is satisfied
\begin{equation}\label{infsup}
\sup_{0\ne (q,\boldsymbol{v})\in \boldsymbol{X}} \frac{|a(p,\boldsymbol u; q,
\boldsymbol{v})|}{\|(q,\boldsymbol{v})\|_{\boldsymbol{X}}} \geq
\gamma_0\|(p,\boldsymbol{u})\|_{\boldsymbol{X}} ,\quad\forall\,
(p,\boldsymbol{u}) \in \boldsymbol{X}.
\end{equation}

\section{The PML problem}

In this section, we introduce the PML formulation for the acoustic-elastic
interaction problem and establish its well-posedness. An error estimate will be
shown for the solutions between the original scattering problem and the
PML problem.

\subsection{PML formulation}

Now we turn to the introduction of an absorbing PML layer. As is shown in
Figure  \ref{fig:geo1}, the domain $\Omega_a$ is surrounded by a PML layer of
thickness $d_j$ which is denoted as $\Omega_{\rm PML}$. Define $\Omega:=
\Omega_a\cup\partial B\cup\Omega_{\rm PML}$. Let $\alpha_j(t)=1+{\rm
i}\sigma_j(t)$ be the PML function which is continuous and satisfies
\[
\sigma_j(t)=0 \quad\text{for} ~|t|<L_j \quad\text{and}\quad
\sigma_j(t)=\sigma_0\left(\frac{|t|-L_j}{d_j}\right)^m \quad\text{otherwise}.
\]
Here $\sigma_0> 0$ is a constant and $m$ is an integer. Following
\cite{cw-motl94}, we introduce the PML by the complex coordinate stretching:
\begin{equation}\label{cs}
 \tilde{x}_j=\int_0^{x_j} \alpha_j(\tau) {\rm d}\tau,\quad 1\leq j\leq 3.
\end{equation}

\begin{figure}
\center
\includegraphics[width=0.3\textwidth]{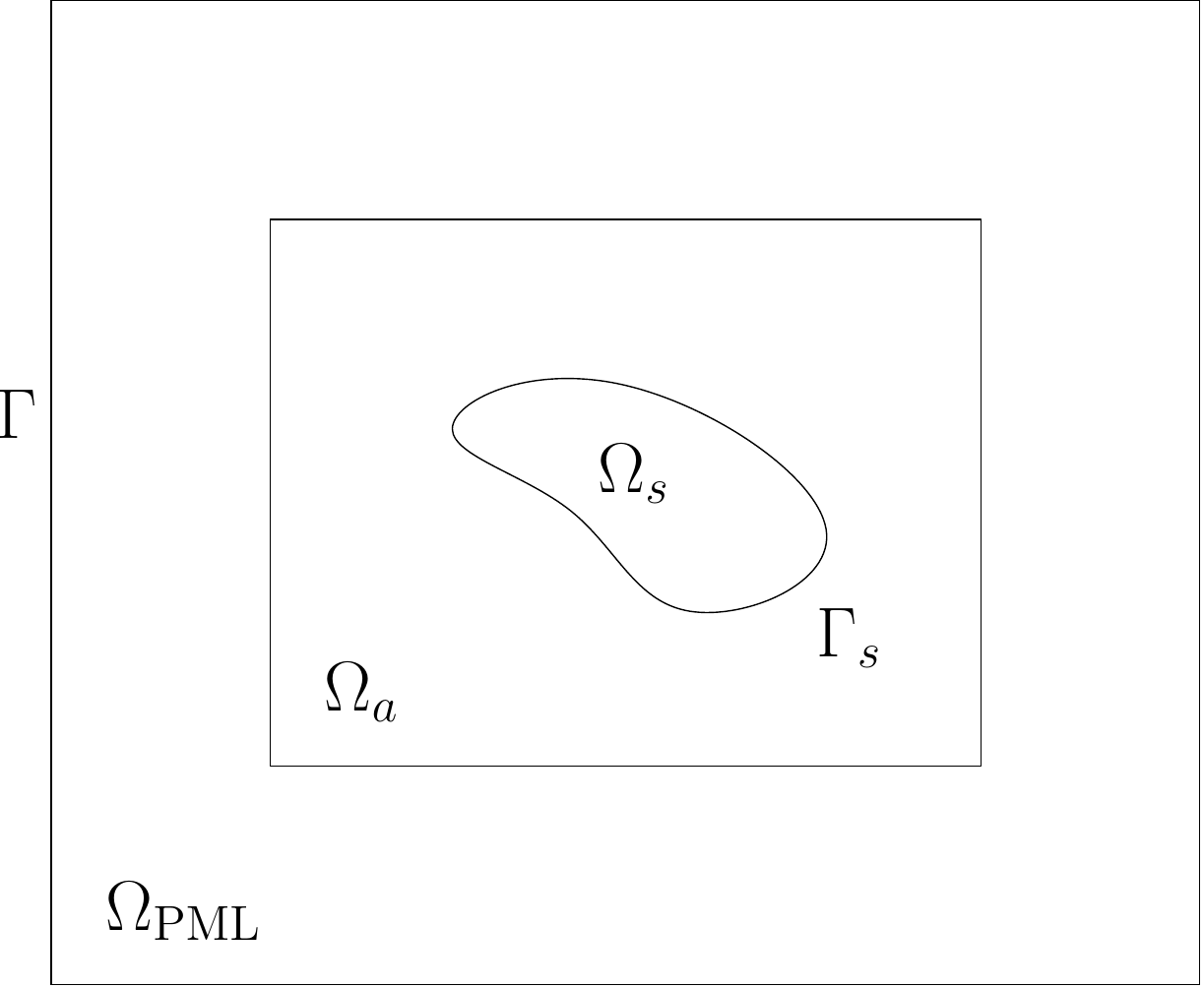}
\caption{A two-dimensional schematic of the geometry for the PML problem.}
\label{fig:geo1}
\end{figure}

Let $\tilde{\boldsymbol x}=(\tilde{x}_1,\tilde{x}_2,\tilde{x}_3)$. Introduce the
new function:
\begin{equation}\label{nf}
 \tilde{p}(\boldsymbol{x})=\begin{cases}
p^{\rm inc}(\boldsymbol{x})+(p(\tilde{\boldsymbol
x})-p^{\rm inc}(\tilde{\boldsymbol x})),\quad&
\boldsymbol{x}\in\Omega_{\rm PML},\\
p(\tilde{\boldsymbol x}) ,\quad
&\boldsymbol{x}\in\Omega_a.
 \end{cases}
\end{equation}
It is clear to note that $\tilde{p}({\boldsymbol
x})=p(\boldsymbol{x})$ in $\Omega_a$ since $\tilde{\boldsymbol
x}=\boldsymbol{x}$ in $\Omega_a$. It can be verified from \eqref{p} and
\eqref{cs} that $\tilde{p}$ satisfies
\[
 \mathscr{L}(\tilde{p}-p^{\rm inc})=0\quad\text{in}~\Omega,
\]
where the PML differential operator is defined by 
\[
  \mathscr{L}p= \nabla\cdot(A\nabla p)+ \kappa^2 b p,
\]
where 
\[
A={\rm
diag}\left(\frac{\alpha_2\alpha_3}{\alpha_1},\frac{\alpha_1\alpha_3}{\alpha_2},
\frac{\alpha_1\alpha_2}{\alpha_3}\right), \quad b=\alpha_1\alpha_2\alpha_3.
\]

It can be verified from \eqref{p} and \eqref{nf} that the outgoing wave
$\tilde{p}(\boldsymbol{x})-p^{\rm inc}(\boldsymbol{x})$
in $\Omega_{\rm PML}$ decays exponentially. Therefore, the homogeneous Dirichlet
boundary condition can be imposed on $\Gamma:=\partial\Omega_{\rm
PML}\setminus\partial B$ to truncate the PML problem. We arrive at the following
truncated PML problem: Find $(\hat p,\hat{\boldsymbol{u}})$ such that
\begin{equation}\label{pmlp}
 \begin{cases}
 \mathscr{L}\hat{p}=g &\quad\text{in} ~ \Omega,\\
 \nabla\cdot\boldsymbol\sigma(\hat{\boldsymbol
u})+\omega^2\hat{\boldsymbol{u}}=0 &\quad\text{in} ~ \Omega_s,\\
 \partial_{\boldsymbol n_1} \hat p =\rho_a\omega^2
\boldsymbol{n}_1\cdot\hat{\boldsymbol{u}},\quad -\hat p
\boldsymbol{n}_1=\boldsymbol{\sigma}(\hat{\boldsymbol{u}})\cdot\boldsymbol n_1
&\quad\text{on} ~\Gamma_s,\\
\hat{p}=p^{\rm inc} &\quad\text{on} ~ \Gamma,
 \end{cases}
\end{equation}
where
\[
 g=\begin{cases}
                 \mathscr{L} p^{\rm inc}&\quad\text{in} ~
\Omega_{\rm PML},\\
                 0&\quad\text{in} ~ \Omega_a.
                \end{cases}
\]

Define 
\[
\boldsymbol{Y}:=H^1(\Omega)\times
H^1(\Omega_s)^3=\{\boldsymbol{\Phi}=(p,\boldsymbol{u}):p\in H^1(\Omega),
\boldsymbol{u}\in H^1(\Omega_s)^3\},
\]
which is endowed with the inner product
\[
(\boldsymbol{\Phi},\boldsymbol{\Psi})_{\boldsymbol{Y}}:=\int_{\Omega}\left(
\nabla
p\cdot\nabla \bar q + p\bar q \right){
\rm d}\boldsymbol x+\int_{\Omega_s}\left( 
\nabla\boldsymbol{u}:\nabla\bar{\boldsymbol
v}+\boldsymbol{u}\cdot\bar{\boldsymbol{v}}\right){\rm
d}\boldsymbol x
\]
for any $\boldsymbol{\Phi}=(p,\boldsymbol{u})$ and
$\boldsymbol{\Psi}=(q,\boldsymbol{v})$. Obviously,
$\|\cdot\|_{\boldsymbol{Y}}=\sqrt{(\cdot,\cdot)_{\boldsymbol{Y}}}$ is a norm on
$\boldsymbol{Y}$.

The weak formulation of the truncated PML problem \eqref{pmlp} reads as
follows: Find $(\hat p,\hat{\boldsymbol{u}})\in \boldsymbol{Y}$ such that
$\hat{p}=p^{\rm inc}$ on $\Gamma$ and
\begin{equation}\label{twp}
 b(\hat p,\hat{\boldsymbol{u}}; q, \boldsymbol{v})=-\int_{\Omega}g\bar q{\rm
d}\boldsymbol{x},\quad\forall\, (q, \boldsymbol{v}) \in \boldsymbol{Y}_{0},
\end{equation}
where
$\boldsymbol{Y}_{0}=\{\boldsymbol{\Phi}=(p,\boldsymbol{u})\in\boldsymbol{Y}:
p=0~\text{on}~\Gamma\}$, and the sesquilinear form $b: \boldsymbol{Y}\times
\boldsymbol{Y}\to\mathbb{C}$ is defined by
\begin{align*}
 b( p,\boldsymbol u; q, \boldsymbol{v})&=\int_\Omega\left( A\nabla p\cdot\nabla
\bar q-\kappa^2 b p\bar q \right){\rm d}\boldsymbol
x+\int_{\Gamma_s}\rho_a\omega^2(\boldsymbol{n}_1\cdot\boldsymbol{u} )\bar q{\rm
d}s \notag \\
&+\int_{\Omega_s}\left(\boldsymbol\sigma(\boldsymbol
u):\nabla\bar{\boldsymbol v}-\omega^2\boldsymbol{u} \cdot\bar { \boldsymbol { v
}}\right){\rm d}\boldsymbol x +\int_{\Gamma_s} (p\boldsymbol{n}_1
)\cdot\bar{\boldsymbol{v}}{\rm d}s. 
\end{align*}

We will reformulate the variational problem \eqref{twp} imposed in the
domain $\Omega\cup\bar\Omega_s$ into an equivalent variational formulation in
the domain $B=\Omega_a\cup\bar\Omega_s$, and discuss the existence and
uniqueness of the weak solution to the equivalent weak formulation. To do so, we
need to introduce the transparent boundary condition for the truncated PML
problem.

\subsection{Transparent boundary condition of the PML problem}

We start by introducing the approximate DtN operator $\mathscr{T}^{\rm PML}:
H^{1/2}(\partial B)\to H^{-1/2}(\partial B)$ associated with
the PML problem. 

Given $\psi\in H^{1/2}(\partial B)$, let $\mathscr{T}^{\rm
PML}\psi=\partial_{\boldsymbol n_2}\phi$ on $\partial B$, where $\phi\in
H^1(\Omega_{\rm PML})$ is the solution of the following boundary value problem
in the PML layer:
\[
 \begin{cases}
\nabla\cdot(A\nabla\phi)+\kappa^2 b\phi=0 &\quad\text{in} ~ \Omega_{\rm PML},\\
\phi=\psi &\quad\text{on} ~ \partial B,\\
\phi=0 &\quad\text{on} ~ \Gamma.
 \end{cases}
\]
The PML problem \eqref{pmlp} can be reduced to the following boundary value
problem: Find $(p^{\rm PML}, \boldsymbol{u}^{\rm PML})$ such that
\begin{equation}\label{cbvp}
 \begin{cases}
 \Delta p^{\rm PML}+\kappa^2 p^{\rm PML}=0 &\quad\text{in} ~ \Omega_a,\\
 \nabla\cdot\boldsymbol\sigma({\boldsymbol u}^{\rm
PML})+\omega^2\boldsymbol{u}^{\rm PML}=0 &\quad\text{in} ~ \Omega_s,\\
 \partial_{\boldsymbol n_1} p^{\rm PML} =\rho_a\omega^2
\boldsymbol{n}_1\cdot\boldsymbol{u}^{\rm PML},\quad -p^{\rm PML}
\boldsymbol{n}_1=\boldsymbol{\sigma}(\boldsymbol{u}^{\rm PML})\cdot\boldsymbol
n_1 &\quad\text{on} ~\Gamma_s,\\
\partial_{\boldsymbol n_2} p^{\rm PML}=\mathscr{T}^{\rm PML}p^{\rm PML} + f^{\rm
PML}&\quad\text{on} ~ \partial B,
 \end{cases}
\end{equation}
where $f^{\rm PML}=\partial_{\boldsymbol n_2} p^{\rm inc}-{\mathscr{T}}^{\rm
PML}p^{\rm inc}$.

The weak formulation of \eqref{cbvp} is to find $(p^{\rm
PML},\boldsymbol{u}^{\rm PML})\in \boldsymbol{X}$ such that
\begin{equation}\label{cwp}
 a^{\rm PML}(p^{\rm PML},\boldsymbol u^{\rm PML}; q, \boldsymbol{v})
 =\int_{\partial B}f^{\rm PML}\bar q{\rm d}s ,\quad\forall\,
(q, \boldsymbol{v}) \in \boldsymbol{X},
\end{equation}
where the sesquilinear form $a^{\rm PML}: \boldsymbol X\times
\boldsymbol X\to\mathbb{C}$ is defined by
\begin{align}\label{csf}
a^{\rm PML}( p,\boldsymbol u; q, \boldsymbol{v})=&\int_{\Omega_a}\left( \nabla
p\cdot\nabla \bar q-\kappa^2 p\bar q \right){\rm d}\boldsymbol
x+\int_{\Gamma_s}\rho_a\omega^2(\boldsymbol{n}_1\cdot\boldsymbol{u} )\bar q{\rm
d}s -\int_{\partial B}(\mathscr{T}^{\rm PML}p) \bar q{\rm d}s\notag\\
 &+\int_{\Omega_s}\left(\boldsymbol\sigma(\boldsymbol u):\nabla\bar{\boldsymbol
v}-\omega^2\boldsymbol{u} \cdot\bar { \boldsymbol { v}} \right) { \rm
d}\boldsymbol x  +\int_{\Gamma_s} (  p\boldsymbol{n}_1
)\cdot\bar{\boldsymbol{v}}{\rm d}s.
\end{align}

The following lemma establishes the relationship between the variational
problem \eqref{cwp} and the weak formulation \eqref{twp}. The proof is
straightforward based on our constructions of the transparent boundary
conditions for the PML problem. The details of the proof is omitted for
simplicity.

\begin{lemm}
Any solution $\hat p$ of the variational problem \eqref{twp} restricted to
$\Omega_a$ is a solution of the variational \eqref{cwp}; conversely,
any solution $p^{\rm PML}$ of the variational problem \eqref{cwp} can be
uniquely extended to the whole domain to be a solution $\hat{p}$ of the
variational problem \eqref{twp} in $\Omega$.
\end{lemm}

\subsection{Convergence of the PML solution}

Now we turn to estimating the error between $(p^{\rm PML}, \boldsymbol u^{\rm
PML})$ and $(p, \boldsymbol u)$. The key is to estimate the error of the
boundary operators $\mathscr{T}^{\rm PML}$ and $\mathscr{T}$.

\begin{lemm}\label{boe}
For any $p, q\in H^1(\Omega_a)$, there exists a constant $C>0$ such that 
\[
|\langle (\mathscr{T}^{\rm PML}-\mathscr{T})p, q\rangle_{\partial B}| \leq
C\alpha_0^3(1+\kappa L)^3e^{-\kappa\gamma_1 \sigma} \|p\|_{L^2(\partial B)}
\|q\|_{L^2(\partial B)},
\]
where $L=\max_{1\leq j\leq 3}L_j,
\alpha_0=\max_{\boldsymbol{x}\in\Gamma}(|\alpha_1(x_1)|,|\alpha_2(x_2)|,
|\alpha_3(x_3)|)$, 
\[
\gamma_1:=\frac{\min_{1\le
j\le3}d_j}{\left(\sum_{j=1}^3(2L_j+d_j)^2\right)^{1/2}},
\]
and $\sigma>0$ is a sufficiently large constant such that $\gamma_1\sigma\ge 1$.
\end{lemm}

\begin{proof}
The proof can follow similar arguments as that in \cite[Theorem
3.8]{bp-mc07}. For the sake of simplicity, we do not elaborate
on the details here.
\end{proof}

\begin{theo}
 Let $\gamma_0$ be the constant in the inf-sup condition
\eqref{infsup}. If 
\[
\gamma_2:=C\alpha_0^3(1+\kappa L)^3
e^{-\kappa\gamma_1\sigma} <\gamma_0, 
\]
then the PML variational problem \eqref{cwp} has a unique weak solution $(p^{\rm
PML},\boldsymbol{u}^{\rm PML})$, which satisfies the error estimate
\begin{align}\label{ee}
 \|(p-p^{\rm PML},\boldsymbol{u}-\boldsymbol{u}^{\rm PML})\|_{\boldsymbol{X}}
\leq \gamma_2 \|p^{\rm PML}-p^{\rm inc}\|_{L^2(\partial B)},
\end{align}
where $(p,\boldsymbol{u})$ is the unique weak solution of the variational problem
\eqref{atbc}.
\end{theo}

\begin{proof}
It suffices to show the coercivity of the sesquilinear form $a^{\rm PML}$
defined in \eqref{csf} in order to prove the unique solvability of the weak
problem \eqref{cwp}. Using Lemma \ref{boe}, and the assumption
$\gamma_2<\gamma_0$, we get for any $(p,\boldsymbol{u}),
(q,\boldsymbol{v})$ in $\boldsymbol{X}$ that
\begin{align*}
 |a^{\rm PML}(p,\boldsymbol{u};q, \boldsymbol{v})|&\geq |a(p,\boldsymbol{u};q, \boldsymbol{v})|
 -\langle (\mathscr{T}^{\rm PML}-\mathscr{T})p, q\rangle_{\partial B}|\\
&\geq|a(p,\boldsymbol{u};q, \boldsymbol{v})|-\gamma_2\|p\|_{H^1(\Omega_a)}
\|q\|_{H^1(\Omega_a)}\\
&\geq\bigl(\gamma_0-\gamma_2\bigr)\|(p,\boldsymbol{u})\|_{\boldsymbol{X}}
\|(q,\boldsymbol{v})\|_{\boldsymbol{X}}.
\end{align*}
It remains to show the error estimate \eqref{ee}. It follows from
\eqref{cwp}--\eqref{csf} that
\begin{align*}
&a(p-p^{\rm PML},\boldsymbol{u}-\boldsymbol{u}^{\rm PML}; q,\boldsymbol{v})\\
=&a(p,\boldsymbol{u};q, \boldsymbol{v})-a(p^{\rm PML},\boldsymbol{u}^{\rm
PML};q,\boldsymbol{v})\\
=&\langle f, q\rangle_{\partial B}-\langle f^{\rm PML},
q\rangle_{\partial B}+a^{\rm PML}(p^{\rm PML},\boldsymbol{u}^{\rm PML};
q,\boldsymbol{v})-a(p^{\rm PML},\boldsymbol{u}^{\rm PML};q, \boldsymbol{v})\\
=&\langle (\mathscr{T}^{\rm PML}-\mathscr{T})p^{\rm inc}, q\rangle_{\partial B}
-\langle (\mathscr{T}^{\rm PML}-\mathscr{T})p^{\rm PML}, q\rangle_{\partial B}\\
=&\langle (\mathscr{T}-\mathscr{T}^{\rm PML})(p^{\rm PML}-p^{\rm inc}),
q\rangle_{\partial B},
\end{align*}
which completes the proof upon using Lemma \ref{boe} and the trace theorem.
\end{proof}

\section{Finite element approximation}

In this section we introduce the finite element approximations of the PML
problem \eqref{twp}. 

\subsection{Error representation formula}

Let $\mathcal{M}_h$ be a regular tetrahedral partition of the domain
$D=\Omega\cup\Gamma_s\cup\Omega_s=\{\boldsymbol x\in\mathbb R^3:
|x_j|<L_j+d_j, 1\leq j\leq 3\}$ such that $\mathcal{M}_h|_\Omega$ and  
$\mathcal{M}_h|_{\Omega_s}$ are also regular tetrahedral partitions of
$\Omega$ and $\Omega_s$, respectively. Let $V_h\subset
H^1(\Omega)$ and $\boldsymbol{U}_h\subset H^1(\Omega_s)^3$ be the conforming
linear finite element space over $\Omega$ and $\Omega_s$, respectively, and
\[
V_{\Gamma,h}=\{p_h\in V_h: p_h=0\;\text{on}~\Gamma\}.
\]
The finite element approximation to the PML problem \eqref{twp} reads as
follows: Find $(p_h, \boldsymbol{u}_h)\in V_h\times \boldsymbol{U}_h$ 
such that $p_h=I_h p^{\rm inc}$ on $\Gamma$ and
\begin{equation}\label{fem-twp}
b(p_h,\boldsymbol u_h; q_h, \boldsymbol{v}_h)=-\int_{\Omega}g\bar q_h{\rm
d}\boldsymbol{x},\quad\forall\, (q_h, \boldsymbol{v}_h) \in V_{\Gamma,h} \times
\boldsymbol{U}_h.
\end{equation}

For any $\varphi\in H^1(\Omega_a)$, let $\tilde\varphi$ be its extension in
$\Omega_{\rm PML}$ such that
\begin{align}
&\nabla\cdot(\bar A\nabla\tilde\varphi)+ \kappa^2\bar{b}\tilde\varphi=0
\quad\text{in} ~\Omega_{\rm PML},\label{5.1a}\\
&\tilde\varphi=\varphi \quad\text{on} ~\partial B, \quad\varphi=0 \quad\text{on}
~\Gamma.\label{5.1b}
\end{align}
Introduce the sesquilinear form $c:H^1(\Omega_{\rm PML})\times H^1(\Omega_{\rm
PML})\to \mathbb{C}$ as follows:
\[
c(\varphi, \psi)=\int_{\Omega_{\rm PML}}\left(\bar A\nabla
\varphi\cdot\nabla\bar\psi-\kappa^2\bar{b} \varphi\bar\psi\right){\rm
d}\boldsymbol x.
\]
The weak formulation for \eqref{5.1a}--\eqref{5.1b} is: Given $\varphi\in
H^{1/2}(\partial B)$, find $\tilde\varphi\in H^1(\Omega_{\rm PML})$
such that $\tilde\varphi=0$ on $\Gamma$, $\tilde\varphi=\varphi$ on $\partial B$, and
\begin{equation}\label{wf-c}
c(\tilde\varphi, \psi)=0,\quad\forall\,\psi\in H^1_0(\Omega_{\rm PML}).
\end{equation}
In this paper we will not elaborate on the well-posedness of \eqref{wf-c} and
simply make the following assumption: There exists a unique solution to
the boundary value problem \eqref{wf-c} in the PML layer.

In order to obtain a constant independent of PML parameter $\sigma$ in the
inf-sup condition, we define
\begin{equation*}
|||\varphi|||_{\Omega_{\rm PML}}=\left(\int_{\Omega_{\rm PML}}
\sum_{j=1}^3\frac{1}{1+\sigma_j}\left|\partial_{x_j}
\varphi\right|^2+(1+\sigma_1\sigma_2\sigma_3)\kappa^2|\varphi|^2\right)^{1/2}.
\end{equation*}
By using the general theory in  \cite[Chap. 5]{ba-73}, we know that there exists
a constant $\hat C > 0$ such that
\begin{equation}\label{ifsp-c}
\sup_{0\ne\psi\in H^1_0(\Omega_{\rm PML})}\frac{|c(\varphi,
\psi)|}{|||\psi|||_{\Omega_{\rm PML}}}\ge \hat C |||\varphi|||_{ \Omega_{\rm
PML}},\quad\forall\,\varphi\in H^1(\Omega_{\rm PML}).
\end{equation}
The constant $\hat C$ depends on the domain $\Omega_{\rm PML}$ and the wave
number $\kappa$.

\begin{lemm}[Estimates for the extension] 
For any $\varphi\in H^1(\Omega_a)$, which is extended to be a function
$\tilde\varphi\in H^1(\Omega)$ according to \eqref{5.1a}--\eqref{5.1b}. Then
there exists a constant $C>0$ independent of $\kappa$ and $\sigma$ such that
\begin{align}
\|\nabla\tilde\varphi\|_{L^2(\Omega_{\rm PML})}&\le C\hat
C^{-1}\alpha_0(1+\kappa L)\|\varphi\|_{H^{1/2}(\partial B)}, \label{est-ext1}\\
\|A\nabla\bar{\tilde\varphi}\cdot\boldsymbol{n}_3\|_{H^{-1/2}(\Gamma)}&\le C\hat
C^{-1}\alpha_0^3 (1+\kappa L)^2\|\varphi\|_{H^{1/2}(\partial
B)},\label{est-ext2}
\end{align}
where $\boldsymbol n_3$ is the unit outward normal vector on $\Gamma$. 
\end{lemm}

\begin{proof}
For any $\zeta\in H^1(\Omega_{\rm PML})$ such that $\zeta=\varphi$ on $\partial
B$ and $\zeta=0$ on $\Gamma$. By the inf-sup condition in \eqref{ifsp-c} and
using \eqref{wf-c}, we know that
\begin{equation*}
\hat C |||\tilde\varphi-\zeta|||_{\Omega_{\rm PML}}\le \sup_{0\ne\psi\in
H^1_0(\Omega_{\rm PML})}\frac{|c(\tilde\varphi-\zeta,
\psi)|}{|||\psi|||_{\Omega_{\rm PML}}} =\sup_{0\ne\psi\in H^1_0(\Omega_{\rm
PML})}\frac{|c(\zeta, \psi)|}{|||\psi|||_{\Omega_{\rm PML}}}.
\end{equation*}
By Cauchy--Schwarz inequality
\begin{equation*}
|c(\zeta, \psi)|\le C\alpha_0^{3/2} (1+\kappa L)\|\zeta\|_{H^1(\Omega_{\rm
PML})}|||\psi|||_{\Omega_{\rm PML}}.
\end{equation*}
Noting 
\[
|||\zeta|||_{\Omega_{\rm PML}}\le C\alpha_0^{3/2}
(1+\kappa L)\|\zeta\|_{H^1(\Omega_{\rm PML})},
\]
using the triangle inequality and the trace inequality, we conclude that
\begin{equation}\label{ext1}
|||\tilde\varphi|||_{\Omega_{\rm PML}}\le C\hat
C^{-1}\alpha_0^{3/2}(1+\kappa L)\|\varphi\|_{H^1(\partial B)},
\end{equation}
which shows the first estimate in the theorem by using the definition of
$|||\cdot|||_{\Omega_{\rm PML}}$.

Next, for any $\psi\in H^1(\Omega_{\rm PML})$ such that $\psi=0$ on $\partial
B$, using \eqref{5.1a} and the integration by parts, we obtain 
\begin{align*}
\int_{\Gamma}(A\nabla\bar{\tilde\varphi}\cdot\boldsymbol n_3)\bar\psi{\rm d}s
&=\int_{\partial\Omega_{\rm PML}}(A\nabla\bar{\tilde\varphi}
\cdot\boldsymbol n_3)\bar\psi{\rm d}s\\
&=\int_{\Omega_{\rm PML}}\left(A\nabla\bar{\tilde\varphi}
\cdot\nabla\bar\psi+\nabla\cdot( A\nabla\bar{\tilde\varphi})\bar\psi\right){\rm
d}\boldsymbol x
=\int_{\Omega_{\rm PML}}\left(A\nabla\bar{\tilde\varphi}\cdot\nabla\bar\psi-
\kappa ^2 b\bar{\tilde\varphi}\bar\psi\right){\rm d}\boldsymbol x.
\end{align*}
It follows from the Cauchy--Schwarz inequality and \eqref{ext1} that 
\begin{align*}
\left|\int_{\Gamma}(A\nabla\bar{\tilde\varphi}
\cdot\boldsymbol n_3)\bar\psi{\rm d}s\right|
&\le C\alpha_0^{3/2}(1+\kappa L)|||\tilde\varphi|||_{\Omega_{\rm
PML}}\|\psi\|_{H^1(\Omega_{\rm PML})}\\
&\le C\hat C^{-1}\alpha_0^{3}(1+\kappa L)^2\|\varphi\|_{H^1(\partial
B)}\|\psi\|_{H^1(\Omega_{\rm PML})},
\end{align*}
which completes the proof after using the trace inequality.
\end{proof}

\begin{lemm}[Error representation formula]
For any $\varphi\in H^1(\Omega_a)$, which is extended to be a function
$\tilde\varphi\in H^1(\Omega)$
according to \eqref{5.1a}--\eqref{5.1b}, and $\varphi_h\in V_{\Gamma,h}$, we have
\begin{align}\label{err-rep}
a(p-p_h,
\boldsymbol{u}-\boldsymbol{u}_h;\varphi,\boldsymbol{v})=&\int_{\Omega}g(\bar{
\tilde\varphi}_h-\bar{\tilde\varphi}){\rm d}\boldsymbol x
-b(p_h,\boldsymbol{u}_h;\tilde\varphi-\tilde\varphi_h,\boldsymbol{v}-\boldsymbol{v}_h) \notag\\
&-\int_{\partial B}(\mathscr{T}-\mathscr{T}^{\rm PML})(p_h-p^{\rm
inc})\bar\varphi{\rm d}s
-\int_{\Gamma}(A\nabla\bar{\tilde{\varphi}}\cdot\boldsymbol n_3)(p^{\rm inc}-I_h
p^{\rm inc}){\rm d}s.
\end{align}
\end{lemm}

\begin{proof}
First by \eqref{asf}, \eqref{atbc}, \eqref{cwp}, and \eqref{csf}, we have
\begin{align}\label{4.7}
a(p-\hat p, \boldsymbol{u}-\hat{\boldsymbol{u}};\varphi,\boldsymbol{v})&=
\int_{\partial B} f\bar\varphi{\rm d}s-\int_{\partial B} f^{\rm
PML}\bar\varphi{\rm d}s +a^{\rm PML}(\hat p,
\hat{\boldsymbol{u}};\varphi,\boldsymbol{v})-a(\hat p,
\hat{\boldsymbol{u}};\varphi,\boldsymbol{v})\notag\\
&=\int_{\partial B}(\mathscr{T}-\mathscr{T}^{\rm PML})(\hat p-p^{\rm
inc})\bar\varphi {\rm d}s.
\end{align}
Using \eqref{4.7} yields 
\begin{align}\label{4.8}
a(p-p_h, \boldsymbol{u}-\boldsymbol{u}_h;\varphi,\boldsymbol{v})=& a(p-\hat p,
\boldsymbol{u}-\hat{\boldsymbol{u}};\varphi,\boldsymbol{v}) + a(\hat p-p_h,
\hat{\boldsymbol{u}}-\boldsymbol{u}_h;\varphi,\boldsymbol{v})\notag\\
=&\int_{\partial B}(\mathscr{T}-\mathscr{T}^{\rm PML})(\hat p-p^{\rm
inc})\bar\varphi{\rm d}s + b(\hat p-p_h,
\hat{\boldsymbol{u}}-\boldsymbol{u}_h;\tilde\varphi,\boldsymbol{v}) \notag\\
&-\int_{\partial B}\mathscr{T}(\hat p-p_h)\bar\varphi{\rm d}s-\int_{\Omega_{\rm
PML}}(A\nabla (\hat p-p_h)\cdot\nabla\bar{\tilde{\varphi}}-\kappa^2 b(\hat
p-p_h)\bar{\tilde{\varphi}}){\rm d}\boldsymbol x.
\end{align}
Recalling that $\boldsymbol n_2$ is the unit outer normal to $\partial B$ which
points outside $B$ and $\boldsymbol{n}_3$ is the unit outer normal
vector on $\Gamma$ directed outside $\Omega_{\rm PML}$,
we deduce that
\begin{align}\label{4.9}
\int_{\Omega_{\rm PML}}(A\nabla(\hat
p-p_h)\cdot\nabla\bar{\tilde{\varphi}}&-\kappa^2 b(\hat
p-p_h)\bar{\tilde{\varphi}}) {\rm
d}\boldsymbol x=\int_{\Gamma}(A\nabla\bar{\tilde{\varphi}}
\cdot\boldsymbol{n}_3)(\hat p-p_h){\rm d}s-\int_{\partial
B}\partial_{\boldsymbol n_2}\bar{\tilde\varphi}(\hat
p-p_h){\rm d}s \notag\\
&=\int_{\Gamma}(A\nabla\bar{\tilde{\varphi}}\cdot\boldsymbol{n}_3)(\hat
p-p_h){\rm d}s-\int_{\partial B}(\mathscr{T}^{\rm PML} (\hat
p-p_h))\bar\varphi{\rm d}s,
\end{align}
where we have used \eqref{5.1a}--\eqref{5.1b}, the definition of
$\mathscr{T}^{\rm PML}$, and the identity (c.f., \cite[Lemma 5.1]{cw-nm08})
\[
\int_{\partial B} (\mathscr{T}^{\rm PML}\varphi) \bar\psi{\rm d}s=\int_{\partial
B} (\mathscr{T}^{\rm PML}\bar\psi)\varphi {\rm d}s, \quad\forall\varphi,\psi\in
H^1(\Omega_{\rm PML}).
\]
By \eqref{twp}, \eqref{fem-twp}, and \eqref{4.8}--\eqref{4.9},
\begin{align*}
&a(p-p_h, \boldsymbol{u}-\boldsymbol{u}_h;\varphi,\boldsymbol{v})  \\
=& b(\hat p-p_h,
\hat{\boldsymbol{u}}-\boldsymbol{u}_h;\tilde\varphi,\boldsymbol{v})-\int_{
\partial B}(\mathscr{T}-\mathscr{T}^{\rm PML})(p_h-p^{\rm inc})\bar\varphi{\rm
d}s-\int_{\Gamma}(A\nabla\bar{\tilde{\varphi}}\cdot\boldsymbol{n}_3)(\hat
p-p_h){\rm d}s  \\
=&\int_{\Omega}g(\bar{\tilde\varphi}_h-\bar{\tilde\varphi}){\rm d}\boldsymbol x
-b(p_h,\boldsymbol{u}_h;\tilde\varphi-\tilde\varphi_h,\boldsymbol{v}-\boldsymbol{v}_h)  \\
&-\int_{\partial B}(\mathscr{T}-\mathscr{T}^{\rm PML})(p_h-p^{\rm
inc})\bar\varphi{\rm d}s
-\int_{\Gamma}(A\nabla\bar{\tilde{\varphi}}\cdot\boldsymbol{n}_3)(p^{\rm
inc}-I_h p^{\rm inc}){\rm d}\boldsymbol x,
\end{align*}
which completes the proof.
\end{proof}

\subsection{A posteriori error analysis}

For any $K\in\mathcal{M}_h$, we denote by $h_K$ its diameter. Let
$\mathcal{B}_h$ denote the set of all sides that do not lie on $\Gamma$. For any
$e\in\mathcal{B}_h$, $h_e$ stands for its length. For any $K\in\mathcal{M}_h$,
we introduce the residual:
\begin{equation}\label{res}
R_K:=\begin{cases}
 \nabla\cdot(A\nabla p_h)+\kappa^2 b p_h -g &\quad\text{for} ~
K\in\mathcal{M}_h|_{\Omega}\\
\nabla\cdot \boldsymbol\sigma(\boldsymbol u_h) +\omega^2\boldsymbol{u}_h
&\quad\text{for} ~ K\in\mathcal{M}_h|_{\Omega_s}
 \end{cases}.
\end{equation}
For any interior side $e\in\mathcal{B}_h$ not lying on the interface $\Gamma_s$
which is the common side of $K_1, K_2\in\mathcal{M}_h$, we define the jump
residual across $e$:
\begin{equation}\label{jmp}
J_e:=\begin{cases}
 (A\nabla p_h)|_{K_1}\cdot\boldsymbol{\nu}-(A\nabla
p_h)|_{K_2}\cdot\boldsymbol{\nu} &\quad\text{for} ~ e\in
\mathcal{B}_h|_{\Omega}\\
\boldsymbol\sigma(\boldsymbol u_h)\cdot\boldsymbol\nu|_{K_1} -
\boldsymbol\sigma(\boldsymbol u_h)\cdot\boldsymbol\nu|_{K_2}
 &\quad\text{for} ~ e\in \mathcal{B}_h|_{\Omega_s}
 \end{cases},
\end{equation}
where we have used the notation that the unit normal vector $\boldsymbol{\nu}$
on $e$ points from $K_2$ to $K_1$. If $e$ lies on the interface $\Gamma_s$, then
we define the jump residual as
\begin{equation}\label{jmp-s}
J_e:=\begin{cases}
 \partial_{\boldsymbol\nu} p_h|_{K_1}-
\rho_a\omega^2\boldsymbol{\nu}\cdot \boldsymbol{u}_h|_{K_2}
 &\quad\text{for} ~ e\subset K_1\in \mathcal{M}_h|_{\Omega}\\
 -p_h\boldsymbol{\nu}|_{K_1}-
\boldsymbol\sigma(\boldsymbol{u}_h)\cdot\boldsymbol\nu|_{K_2}
 &\quad\text{for} ~ e\subset K_2\in \mathcal{M}_h|_{\Omega_s}
 \end{cases},
\end{equation}
For any $K\in\mathcal{M}_h$, we define the local error estimator $\eta_K$ as
\[
\eta_K:=\Biggl(\|h_KR_K\|_{L^2(K)}^2+\frac{1}{2}\sum_{e\subset\partial
K\setminus\Gamma_s}h_e\|J_e\|_{L^2(e)}^2 + \sum_{e\subset\partial
K\cap\Gamma_s}h_e\|J_{e}\|_{L^2(e)}^2\Biggr)^{1/2}.
\]

\begin{theo}\label{thmp}
There exists a constant $C>0$ depending only on $\gamma_1$ and the minimum angle
of the mesh $\mathcal{M}_h$ such that the following a posterior error estimate
holds
\begin{align*}
&\|p-p_h\|_{H^1(\Omega_a)}+\|\boldsymbol{u}-\boldsymbol{u}_h\|_{H^1(\Omega_s)^3}
\le C\hat C^{-1}\alpha_0^3(1+\kappa L)\Biggl(\sum_{K\in\mathcal{M}_h}
\eta_K^2\Biggr)^{1/2} \notag \\
&+C\hat C^{-1}\alpha_0^3(1+\kappa L)^3 e^{-\gamma_1 \kappa\sigma}\|p_h-
p^{\rm inc}\|_{H^{1/2}(\partial B)} +C\hat C^{-1}\alpha_0^3(1+\kappa L)^2 \|
p^{\rm inc}- I_h p^{\rm inc}\|_{H^{1/2}(\Gamma)}.
\end{align*}
\end{theo}

\begin{proof}
Let $\Pi_h:H^1_{\Gamma}(\Omega)\to V_{\Gamma,h}$
and $\boldsymbol{\Pi}_h:H^1(\Omega_s)^3\to \boldsymbol{U}_h$ be Scott--Zhang
\cite{sz-mc90} interpolation operators satisfying the following interpolation
estimates: For any $\varphi\in H^1(\Omega)$ and
$\boldsymbol{v}\in H^1(\Omega_s)^3$,
\begin{equation}\label{pi-p}
 \begin{cases}
\|\varphi-\Pi_h\varphi\|_{L^2(K)}\le Ch_K \|\nabla\varphi\|_{L^2(\tilde K)^3}\\
\|\varphi-\Pi_h\varphi\|_{L^2(e)}\le Ch_K^{1/2} \|\nabla\varphi\|_{L^2(\tilde
e)^3}
 \end{cases}
 \quad\text{for} ~ K\in \mathcal{M}_h|_{\Omega}
\end{equation}
and
\begin{equation}\label{pi-u}
 \begin{cases}
  \|\boldsymbol{v}-\boldsymbol{\Pi}_h\boldsymbol{v}\|_{L^2(K)}\le Ch_K
\|\nabla\boldsymbol{v}\|_{L^2(\tilde K)^{3\times 3}}\\
\|\boldsymbol{v}-\boldsymbol{\Pi}_h\boldsymbol{v}\|_{L^2(e)}\le Ch_K^{1/2}
\|\nabla\boldsymbol{v}\|_{L^2(\tilde e)^{3\times 3}}
 \end{cases}
 \quad\text{for} ~ K\in \mathcal{M}_h|_{\Omega_s},
\end{equation}
where $\tilde K$ and $\tilde e$ are the union of all elements in $\mathcal{M}_h$
having a non-empty intersection with $K\in\mathcal{M}_h$ and the side $e$,
respectively.

Taking $\tilde\varphi_h=\Pi_h\tilde\varphi\in V_{\Gamma,h}$ and
$\boldsymbol{v}_h= \boldsymbol{\Pi}_h\boldsymbol{v}\in\boldsymbol{U}_h$ in the
error representation formula \eqref{err-rep}, we get 
\begin{align}\label{I4}
&a(p-p_h, \boldsymbol{u}-\boldsymbol{u}_h;\varphi,\boldsymbol{v})\notag\\
=&\int_{\Omega}g(\overline{\Pi_h\tilde\varphi-\tilde\varphi}){\rm d}\boldsymbol
x -b(p_h,\boldsymbol{u}_h;\tilde\varphi-\Pi_h\tilde\varphi,\boldsymbol{v}
-\boldsymbol{\Pi}_h\boldsymbol{v})\notag\\
&-\int_{\partial B}(\mathscr{T}-\mathscr{T}^{\rm PML})(p_h-p^{\rm
inc})\bar\varphi{\rm d}s-\int_{\Gamma}(A\nabla\bar{\tilde{\varphi}}
\cdot\boldsymbol{n}_3)(p^{\rm inc}-I_h p^{\rm inc}){\rm d}s \notag\\
=&I_1+I_2+I_3+I_4.
\end{align}
It follows from the integration by parts and \eqref{res}--\eqref{jmp-s} that 
\begin{align*}
I_1+I_2=&\sum_{K\in\mathcal{M}_h|_{\Omega}}\Biggl(\int_K
R_K(\overline{\tilde\varphi-\Pi_h\tilde\varphi}){\rm d}\boldsymbol x
+\frac{1}{2}\sum_{e\subset\partial K\backslash\Gamma_s}\int_{e}
J_e(\overline{\tilde\varphi-\Pi_h\tilde\varphi}){\rm d}s\\
&+\sum_{e\subset\partial K\cap\Gamma_s}\int_{e}
J_{e}(\overline{\tilde\varphi-\Pi_h\tilde\varphi}){\rm d}s\Biggr)
+\sum_{K\in\mathcal{M}_h|_{\Omega_s}}\Biggl(\int_K
R_K\cdot(\overline{\boldsymbol{v}-\boldsymbol{\Pi}_h\boldsymbol{v}}){\rm
d}\boldsymbol x\\
&+\frac{1}{2}\sum_{e\subset\partial K\backslash\Gamma_s}\int_{e}
J_e\cdot(\overline{\boldsymbol{v}-\boldsymbol{\Pi}_h\boldsymbol{v}}){\rm d}s
+\sum_{e\subset\partial K\cap\Gamma_s}\int_{e}
J_{e}\cdot(\overline{\boldsymbol{v}-\boldsymbol{\Pi}_h\boldsymbol{v}}){\rm
d}s\Biggr).
\end{align*}
By \eqref{pi-p}--\eqref{pi-u} and the estimate \eqref{est-ext1}, we have
\begin{align*}
|I_1+I_2|&\le C\Biggl(\sum_{K\in\mathcal{M}_h}\eta_K^2\Biggr)^{1/2}
\|\nabla\tilde\varphi\|_{L^{ 2 }(O)}\\
&\le C\hat C^{-1}\alpha_0^3(1+\kappa
L)\Biggl(\sum_{K\in\mathcal{M}_h}\eta_K^2\Biggr)^{1/2}
\|\varphi\|_{H^{1/2}(\partial B)}.
\end{align*}
By Lemma \ref{boe}, we have
\begin{equation*}
|I_3|\le C\hat C^{-1}\alpha_0^3(1+\kappa L)^3e^{-k\gamma_1\sigma}\|p_h-
p^{\rm inc}\|_{H^{1/2}(\partial B)}\|\varphi\|_{H^{1/2}(\partial B)}.
\end{equation*}
It follows from \eqref{est-ext2} that 
\begin{equation*}
|I_4|\le C\hat C^{-1}\alpha_0^3(1+\kappa L)^2\|\varphi\|_{H^{1/2}(\partial
B)}\|p^{\rm inc}-I_h p^{\rm inc}\|_{H^{1/2}(\Gamma)}.
\end{equation*}
The proof is completed by using the above estimates in \eqref{I4} and the
inf-sup condition \eqref{infsup}.
\end{proof}

\section{Numerical experiments}

According to the discussion in section 4, we choose the PML medium property
as the power function and need to specify the thickness $d_j$ of the layers and
the medium parameter $\sigma$. It is clear to note from Theorem \ref{thmp} that
the a posteriori error estimate consists of two parts: the PML error
$\epsilon_{\rm PML}$ and the finite element discretization error
$\epsilon_{\rm FEM}$, where
\begin{align}
\label{eta:fem}&\epsilon_{\rm FEM}=  \Biggl(\sum_{K\in\mathcal{M}_h}
\eta_K^2\Biggr)^{1/2} + \| p^{\rm inc}- I_h p^{\rm inc}\|_{H^{1/2}(\Gamma)},\\
\label{eta:pml}&\epsilon_{\rm PML}=\alpha_0^3(1+\kappa L)^3 e^{-\gamma_1
\kappa \sigma}\|p_h- p^{\rm inc}\|_{H^{1/2}(\partial B)}.
\end{align}
In our implementation, we first choose $d_j$ and $\sigma$ such that
$\alpha_0^3(1+\kappa L)^3 e^{-\gamma_1 \kappa \sigma}\le 10^{-8}$, which
makes the PML error \eqref{eta:pml} negligible compared with the finite element
discretization error \eqref{eta:fem}. Once the PML region and the medium
property are fixed, we use the standard finite element adaptive strategy to
modify the mesh according to the a posteriori error estimate. For any
$K\in\mathcal{M}_h$, we define the local a posteriori error estimator
\[
\hat\eta_K = \eta_K + \|p^{\rm inc}- I_h p^{\rm inc}\|_{H^{1/2}(\Gamma\cap\partial K)}.
\]
The adaptive FEM algorithm is summarized in Table \ref{alg}.

\begin{table}
\caption{The adaptive FEM algorithm.}
\vskip -15pt
\hrulefill

\begin{tabular}{ll}
 1 & Given a tolerance $\epsilon > 0$ and mesh refinement threshold $\tau\in
(0,1)$;\\
 2 & Choose $d_j$ and $\sigma$ such that $\alpha_0^3(1+\kappa L)^3 e^{-\gamma_1
\kappa\sigma}< 10^{-8}$;\\
 3 & Construct an initial tetrahedral partition $\mathcal{M}_h$ over $D$ and
compute error estimators;\\
 4 &  While $\epsilon_h>\epsilon$ do\\
 5 & \qquad choose  $\hat{\mathcal{M}}_h\subset\mathcal{M}_h$ according to the
strategy $\eta_{\hat{\mathcal{M}}_h}>\tau\eta_{\mathcal{M}_h}$;\\
 6 & \qquad refine all the elements in $\hat{\mathcal{M}}_h$ and obtain a new
mesh denoted still by $\mathcal{M}_h$;\\
 7 & \qquad solve the discrete problem \eqref{fem-twp} on the new mesh
$\mathcal{M}_h$;\\
 8 & \qquad compute the corresponding error estimators;\\
 9 & End while.
\end{tabular}

\hrulefill
\label{alg}
\end{table}

In the following, we present two examples to demonstrate the competitive
numerical performance of the proposed algorithm. The first-order linear element
is used for solving the problem. Our implementation is based on parallel
hierarchical grid (PHG) \cite{phg}, which is a toolbox for developing parallel
adaptive finite element programs on unstructured tetrahedral meshes. The linear
system resulted from finite element discretization is solved by the PCG solver.

{\bf Example 1.} We consider a problem with an exact solution. We set
the elastic region $\Omega_s:=B(0,0.2)$ and the acoustic region
$\Omega_a:=B(0,0.5)\setminus\bar\Omega_s$, where $B(0,R):=\{\boldsymbol{x}\in
\mathbb{R}^3: |\boldsymbol{x}|<R\}$ denotes the
ball with radius $R>0$ and centering at the origin. Let 
\begin{equation}\label{ex:s}
p(\boldsymbol x)=\frac{e^{{\rm i} \kappa |\boldsymbol x-\boldsymbol
x_0|}}{|\boldsymbol x-\boldsymbol x_0|} \quad\text{and}\quad
\boldsymbol{u}(\boldsymbol x)=\omega^2\nabla p(\boldsymbol x),
\end{equation}
where $\boldsymbol x_0=(1, 0, 0)^\top$. The parameters are chosen as
$\kappa=1$, $\omega=1$, $\lambda=0.5$, $\mu=0.25$, and $\rho_a=1$ such that
\begin{equation}\label{af}
\kappa^2(\lambda+2\mu)=\omega^2.
\end{equation}

First it is easy to verify that 
\[
\Delta p+\kappa^2 p =0\quad\text{in} ~ \Omega_a.
\]
When $\mu$ and $\lambda$ are constants, the Navier equation \eqref{une} reduces
to 
\begin{equation}\label{cne}
 \mu\Delta\boldsymbol u+(\lambda+\mu)\nabla\nabla\cdot\boldsymbol
u+\omega^2\boldsymbol u=0\quad\text{in}~ \Omega_s.
\end{equation}
Using \eqref{ex:s} and \eqref{af}, we have from a straightforward calculation
that
\begin{align*}
\mu\Delta\boldsymbol u+(\lambda+\mu)\nabla\nabla\cdot\boldsymbol
u+\omega^2\boldsymbol u &=\omega^2\left(\mu\nabla\cdot\nabla(\nabla p) +
(\lambda+\mu)\nabla (\Delta p)+\omega^2\nabla p\right)\\
&=\omega^2\left(\mu\nabla\cdot\nabla(\nabla p) -\kappa^2 (\lambda+\mu)\nabla
p+\omega^2\nabla p\right)\\
&=\omega^2\left(-\kappa^2\mu (\nabla p) -\kappa^2 (\lambda+\mu)\nabla
p+\omega^2\nabla p\right)\\
&=\omega^2\left(-\kappa^2(\lambda+2\mu)+\omega^2\right)\nabla p=0.
\end{align*}
which shows that $\boldsymbol{u}=\omega^2\nabla p$ satisfies \eqref{cne} in
$\Omega_s$. It can be verified that the interface conditions
\eqref{inc1}--\eqref{inc2} are also satisfied by letting $\rho_a=1$.

Let $q=p|_{\partial B(0, 0.5)}$ and consider the following acoustic-elastic
interaction problem with the Dirichlet boundary condition:
\[
 \begin{cases}
  \Delta p+\kappa^2 p=0&\quad\text{in} ~ B(0, 0.5)\setminus\bar{B}(0, 0.2),\\
  \mu\Delta\boldsymbol u+(\lambda+\mu)\nabla\nabla\cdot\boldsymbol
u+\omega^2\boldsymbol u=0&\quad\text{in} ~ B(0, 0.2),\\
p=q&\quad\text{on} ~ \partial B(0, 0.5).
 \end{cases}
\]
We may test the adaptive FEM algorithm by solving the above boundary value
problem. 

Figure \ref{ex1:err} displays the errors of $p$ and $\boldsymbol u$ against
the number of nodal points $N_p$ in $B(0, 0.5)\setminus\bar{B}(0, 0.2)$
and $N_{\boldsymbol u}$ in $B(0, 0.2)$, respectively. It clearly shows that the
adaptive FEM yields quasi-optimal convergence rates, i.e., 
\[
\|p-p_h\|_{H^1(\Omega_a)}=O(N_p^{-1/3}),\quad \eta_{p,h}=O( 
N_p^{-1/3})
\]
and
\[
\|\boldsymbol{u}-\boldsymbol{u}_h\|_{\boldsymbol{H}^1(\Omega_s)}=O(
N_{\boldsymbol{u}}^{-1/3}), \quad \eta_{\boldsymbol{u},h}=O(
N_{\boldsymbol{u}}^{-1/3}),
\]
where $\eta_{p,h}$ and $\eta_{\boldsymbol{u},h}$ are the a posterior error
estimators for $p$ and $\boldsymbol u$, respectively. Figure \ref{ex1:meshp}
plots the adaptive mesh of $\Omega_a$ for solving $p_h$ and Figure \ref{ex1:sp}
plots the mesh on a cross section of the domain $\Omega_a$ on the $xz$-plane.
Figure \ref{ex1:meshu} plots the adaptive mesh of $\Omega_s$ for solving
$\boldsymbol{u}_h$ and Figure \ref{ex1:su} plot the mesh on the cross section of
the domain $\Omega_s$ on the $xz$-plane.

\begin{figure}
\includegraphics[width=0.8\textwidth]{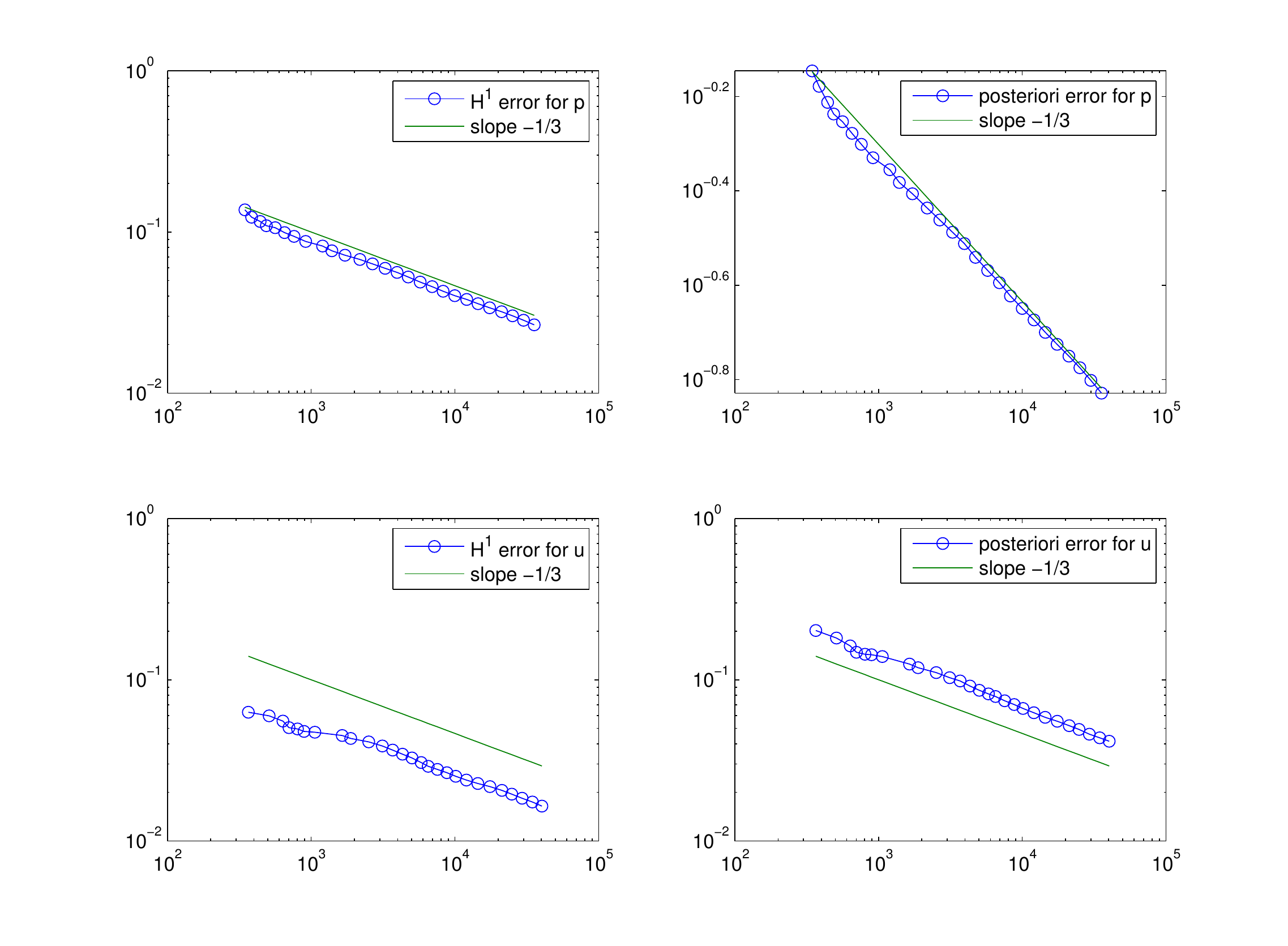}
\caption{Example 1: Quasi-optimality of $H^1$- error estimates and the a
posteriori error estimates.}\label{ex1:err}
\end{figure}

\begin{figure}
\includegraphics[width=0.5\textwidth]{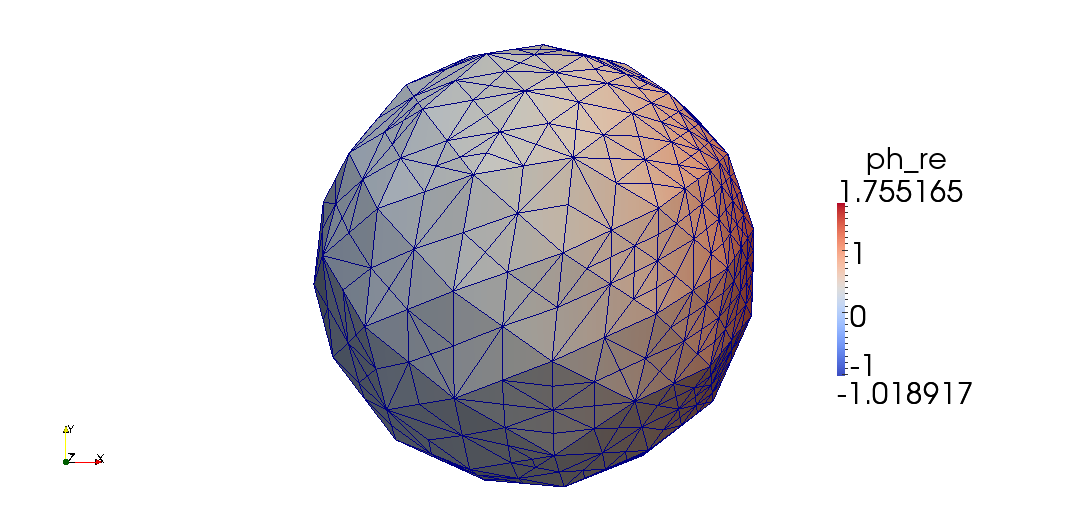}
\caption{Example 1: An adaptive mesh with 20390 elements of
$\Omega_a$.}\label{ex1:meshp}
\end{figure}

\begin{figure}
\includegraphics[width=0.5\textwidth]{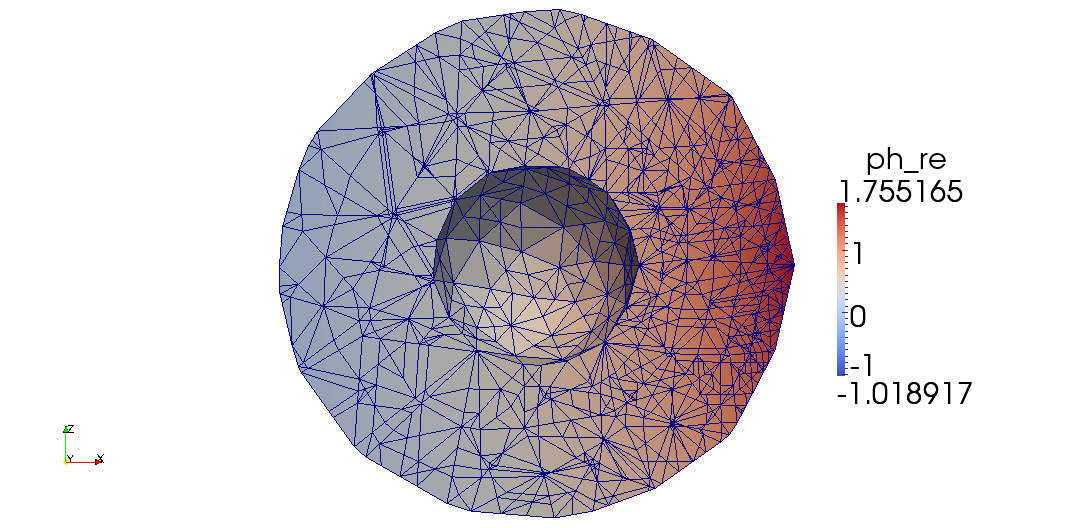}
\caption{Example 1: The cross section of the mesh in Figure \ref{ex1:meshp} on
the $xz$-plane.}\label{ex1:sp}
\end{figure}

\begin{figure}
\includegraphics[width=0.5\textwidth]{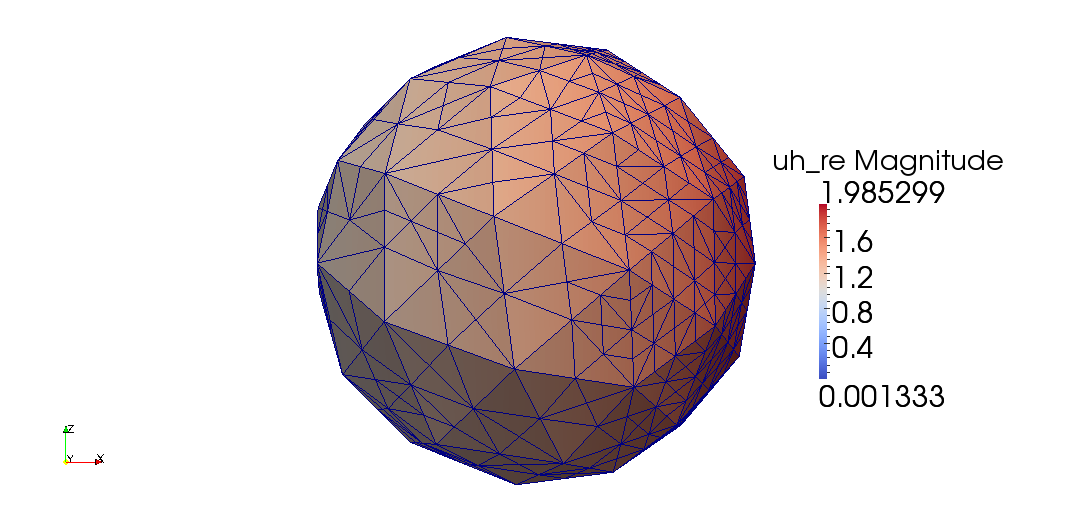}
\caption{Example 1: An adaptive mesh with 7655 elements of
$\Omega_s$.}\label{ex1:meshu}
\end{figure}

\begin{figure}
\includegraphics[width=0.5\textwidth]{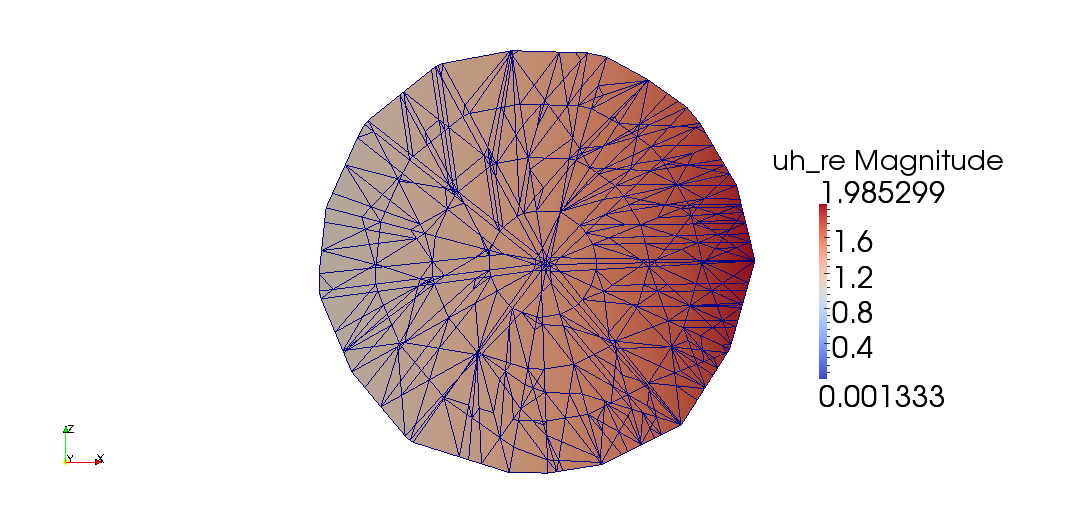}
\caption{Example 1: The cross section of the mesh in Figure \ref{ex1:meshu} on 
$xz$-plane.}\label{ex1:su}
\end{figure}

{\bf Example 2.} This example concerns the scattering of the incident plane wave
\[
p^{\rm inc}(\boldsymbol x)=e^{-{\rm i}\kappa x_3}. 
\]
The Dirichlet boundary condition on the PML layer outer boundary $\Gamma$ is set
by $p=p^{\rm inc}$. We choose $\kappa=2$, $\omega=2\pi$, $\lambda=1$, $\mu=2$,
and $\rho_a=1$. Let the elastic region and the acoustic region be $\Omega_s=
B_1\backslash\bar B_0$ and $\Omega_a = B_2 \backslash \bar \Omega_s$,
respectively. Here $B_0=(-0.1,0.1)\times(-0.1,0.1)\times (-0.2,0.0), B_1
=(-0.2,0.2)\times(-0.2,0.2)\times (-0.2,0.2)$, and
$B_2=[-0.6,0.6]\times[-0.6,0.6]\times[-0.6,0.6]$. The PML domain is $\Omega_{\rm
PML} = (0, 1)\times(0, 1)\times(0, 1)\setminus\bar{B}_2 $, i.e., the thickness
of the PML layer is 0.4 in each direction. In this example, the elastic solid
is a rectangular box with a small rectuangular dent on the surface. The
solutions of $p$ and $\boldsymbol u$ may have singularities around the corners
of the dent. We choose $\sigma = 16$ and $m = 2$ for the medium property to
ensure the PML error is negligible compared to the finite element error.

For this example, we set the numerical solution on the very fine mesh to be a
reference solution since there is no analytic solution. Figure \ref{ex2:err}
shows the errors of $p$ and $\boldsymbol u$ against the number of nodal points
$N_p$ and $N_{\boldsymbol u}$. It is clear to note that the FEM algorithm yields
a quasi-optimal convergence rate. The surface plots of the amplitude of the
fields are shown as follows: Figure \ref{ex2:p} shows the real part of $p_h$ for
the cross section in $\Omega_a$ on the $yz$-plane and Figure \ref{ex2:u} shows
the real part of $\boldsymbol{u}_h$ for the cross section in $\Omega_s$ on the
$yz$-plane.

\begin{figure}
\includegraphics[width=0.8\textwidth]{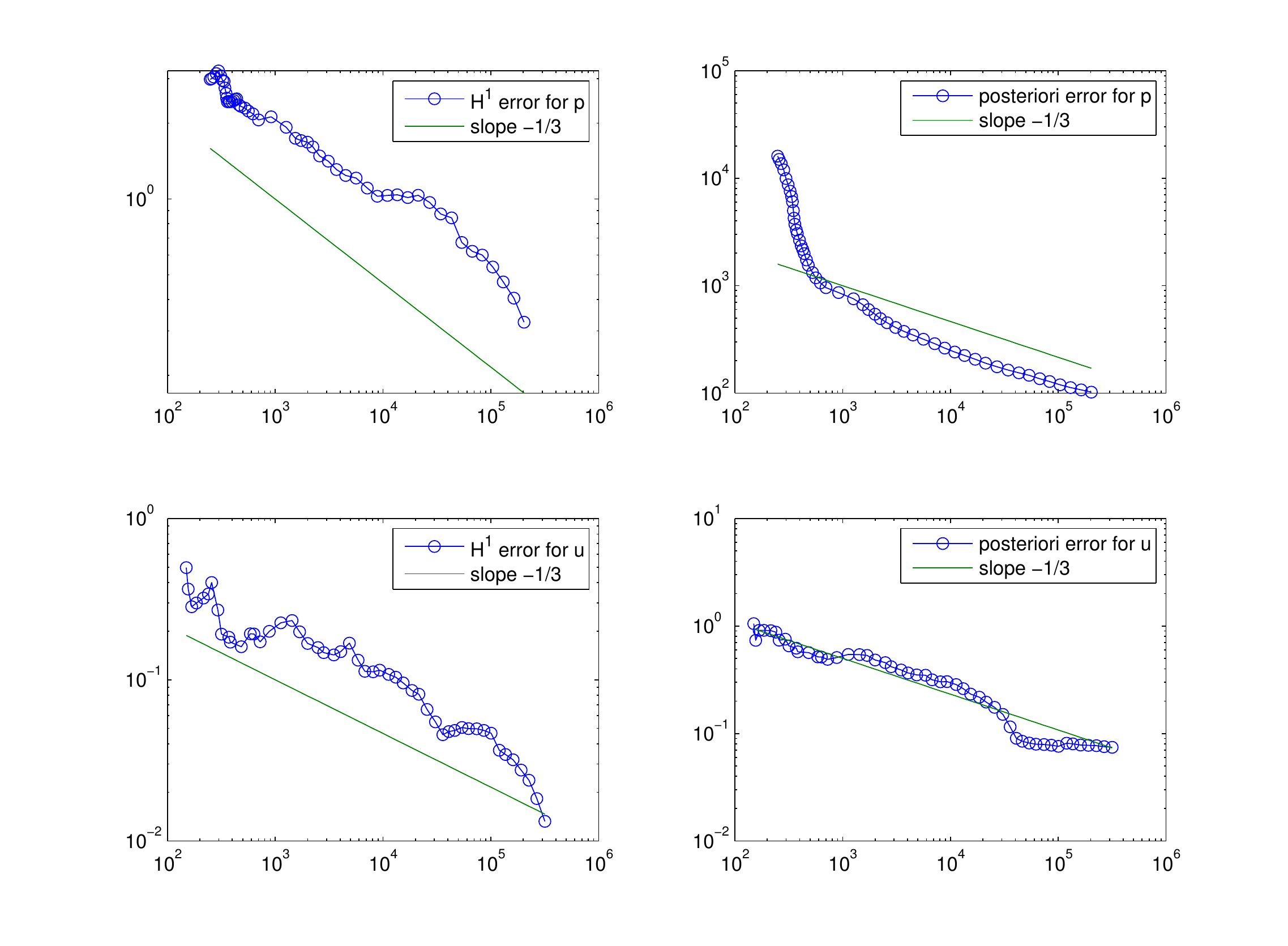}
\caption{Example 2: Quasi-optimality of $H^1$- error estimates and the a
posteriori error estimates.}\label{ex2:err}
\end{figure}

\begin{figure}
\includegraphics[width=0.5\textwidth]{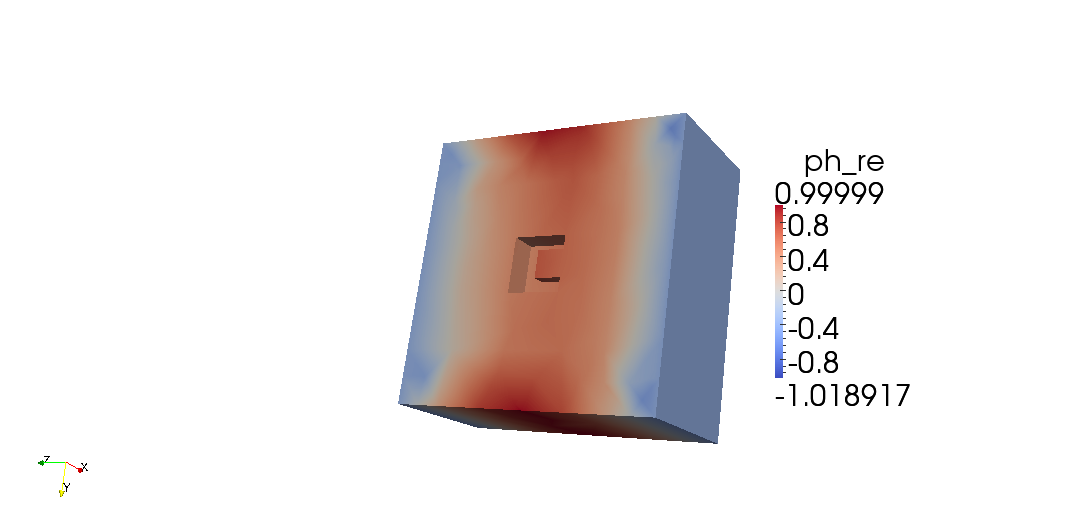}
\caption{Example 2: The amplitude of the real part of $p_h$ for the cross
section of $\Omega_a$ on the $yz$-plane.}\label{ex2:p}
\end{figure}

\begin{figure}
\includegraphics[width=0.5\textwidth]{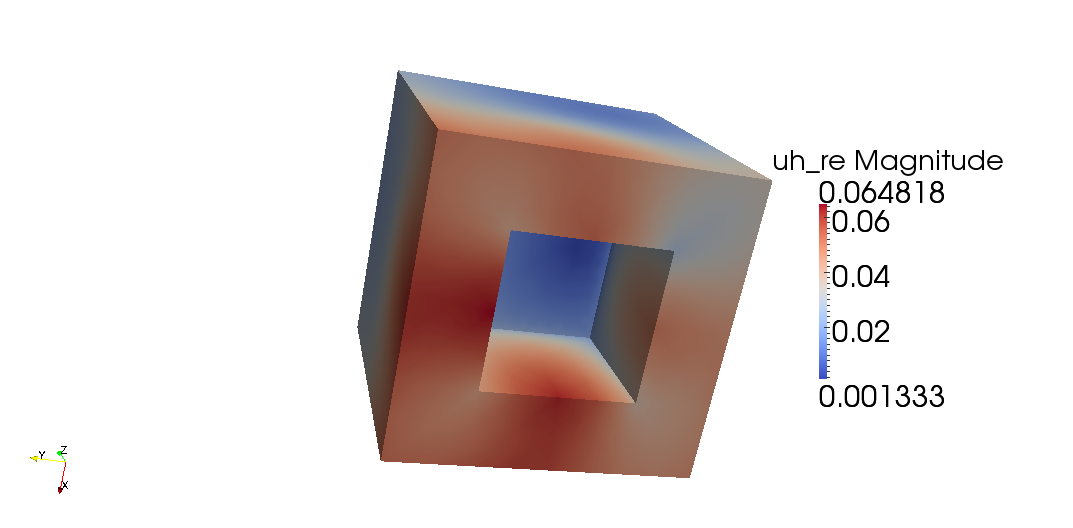}
\caption{Example 2: The amplitude of the real part of $\boldsymbol{u}_h$ for the
cross section of $\Omega_s$ on the $yz$-plane.}\label{ex2:u}
\end{figure}

\section{Concluding remarks}

We have studied a variational formulation for the acoustic-elastic interaction
problem in $\mathbb{R}^3$ and adopted the PML to truncate the unbounded
physical domain. The scattering problem is reduced to a boundary value problem
by using transparent boundary conditions. We prove that the truncated PML
problem has a unique weak solution which converges exponentially to the solution
of the original problem by increasing the PML parameters. We incorporate the
adaptive mesh refinement with a posteriori error estimate for the finite element
method to handle the problem where the solution may have singularities.
Numerical results show that the proposed method is effective to solve the
acoustic-elastic interaction problem.

\end{document}